\documentclass[12pt]{amsart}
\usepackage[margin=0.7in]{geometry}

\usepackage[english]{babel}
\usepackage{amsmath,amssymb,amsfonts,amsthm,enumerate}
\usepackage{url}
\usepackage{graphicx,epstopdf,color}

\numberwithin{equation}{section}

\newtheorem{theorem}{Theorem}[section]
\newtheorem{lemma}[theorem]{Lemma}
\newtheorem{definition}[theorem]{Definition}

\newtheorem{proposition}[theorem]{Proposition}
\newtheorem{corollary}[theorem]{Corollary}

\newcommand{\R}{\mathbb{R}}

\newcommand{\N}{\mathbb{N}}

\renewcommand{\epsilon}{\varepsilon}
\newcommand{\eps}{\varepsilon}
\newcommand{\e}{\varepsilon}

\renewcommand{\le}{\leqslant}

\renewcommand{\ge}{\geqslant}

\title[A three-dimensional symmetry result]{A three-dimensional symmetry result\\
for a phase transition equation\\ in the genuinely
nonlocal regime}

\author{Serena Dipierro}
\address[Serena Dipierro]{Dipartimento di Matematica,
Universit\`a degli studi di Milano,
Via Saldini 50, 20133 Milan, Italy}
\email{serena.dipierro@unimi.it}

\author{Alberto Farina}
\address[Alberto Farina]{LAMFA -- CNRS UMR 6140 and
Facult\'e des Sciences, Universit\'e de Picardie Jules Verne,
33 rue Saint-Leu,
80039 Amiens CEDEX 1, France}
\email{alberto.farina@u-picardie.fr}

\author{Enrico Valdinoci}
\address[Enrico Valdinoci]{School of Mathematics and Statistics,
University of Melbourne,
813 Swanston Street, Parkville VIC 3010, Australia, and
Istituto di Matematica Applicata e Tecnologie 
Informatiche, Consiglio Nazionale delle Ricerche, Via Ferrata 1, 27100
Pavia, Italy,
and Dipartimento di Matematica, Universit\`a degli studi di Milano,
Via Saldini 50, 20133 Milan, Italy}
\email{enrico@mat.uniroma3.it}

\begin{document}

\begin{abstract} We consider bounded solutions of the nonlocal Allen-Cahn equation
$$ (-\Delta)^s u=u-u^3\qquad{\mbox{ in }}\R^3,$$
under the monotonicity condition~$\partial_{x_3}u>0$
and in the 
genuinely
nonlocal regime in which~$s\in\left(0,\frac12\right)$.

Under the limit assumptions
$$ \lim_{x_n\to-\infty} u(x',x_n)=-1\quad{\mbox{ and }}\quad
\lim_{x_n\to+\infty} u(x',x_n)=1,$$
it has been recently shown in~\cite{DSVarxiv} that~$u$ is necessarily $1$D,
i.e. it depends only on one Euclidean variable.

The goal of this paper is to obtain a similar result without
assuming such limit conditions.

This type of results can be seen as nonlocal counterparts of the celebrated
conjecture formulated by Ennio De Giorgi in~\cite{DG}.
\end{abstract}

\maketitle

\section{Introduction}

Goal of this paper is to provide a
one-dimensional symmetry result for a phase transition
equation in a genuinely nonlocal regime in three spatial dimensions.
That is, we consider a fractional Allen-Cahn equation of the type
\begin{equation}\label{ALLEN}
(-\Delta)^s u=u-u^3\end{equation} in~$\R^3$, with
$s\in\left(0,\frac12\right)$, and, under boundedness and monotonicity
assumptions, we prove that~$u$ depends only on one variable, up to a rotation.

In this setting,
as customary, for~$s\in(0,1)$, we consider the fractional Laplace operator
defined by
$$ (-\Delta)^s u(x):=c_{n,s}\int_{\R^n}\frac{2u(x)-u(x+y)-u(x-y)}{|y|^{n+2s}}\,dy,$$
with
\begin{equation}\label{cns} c_{n,s}:=\frac{2^{2s-1}\,s\,\Gamma\left( \frac{n}2+s\right)}{
\pi^{\frac{n}2}\,\Gamma(1-s)},\end{equation}
being~$\Gamma$ the Euler's Gamma Function.

Moreover, we say that~$u$ is $1$D if there exist~$u_o:\R\to\R$ and~$\omega_o\in
S^{n-1}$ such that~$u(x)=u_o(\omega_o\cdot x)$ for any~$x\in\R^n$.
Then, our main result in this paper is the following:

\begin{theorem}\label{MAIN}
Let~$n\le3$, $s\in\left(0,\frac12\right)$ and~$u\in C^2(\R^n,[-1,1])$
be a solution of~$(-\Delta)^s u=u-u^3$ in~$\R^n$, 
with~$\partial_{x_n}u>0$ in~$\R^n$.
Then, $u$ is $1$D.
\end{theorem}

Recently, a result similar to that in Theorem~\ref{MAIN} has been established
in Theorem~1.4 of~\cite{DSVarxiv}, under the additional assumption that
\begin{equation} \label{LIMIT}
\lim_{x_n\to-\infty} u(x',x_n)=-1\quad{\mbox{ and }}\quad
\lim_{x_n\to+\infty} u(x',x_n)=1.\end{equation}
Therefore, Theorem~\ref{MAIN} here
is the extension of Theorem~1.4 of~\cite{DSVarxiv}
in which it is not necessary to assume the limit condition~\eqref{LIMIT}.\medskip

We recall that equation~\eqref{ALLEN}
represents a phase transition subject to long-range interactions, see e.g.
Chapter~5 in~\cite{bucur} for a detailed description of the model.
In particular, the states~$u=-1$ and~$u=1$ would correspond to the
``pure phases'' and equation~\eqref{ALLEN} models
the coexistence between intermediate phases and studies
the separation between them.  

At a large scale, the separation between phases
is governed by the minimization of a limit interface, which can be either
of local or nonlocal type,
in dependence of the fractional parameter~$s\in(0,1)$,
with a precise bifurcation occurring at
the threshold~$s=\frac12$. 
More precisely, as proved in~\cite{MR2948285, MR3133422},
if~$u$ is a local energy minimizer for equation~\eqref{ALLEN}
and~$u_\eps(x):=u(x/\eps)$, as~$\eps\searrow0$
we have that~$u_\eps$ approaches a ``pure phase''
step function with values in~$\{-1,1\}$. That is, we can write, up to subsequences,
$$ \lim_{\eps\searrow0} u_\eps = \chi_E-\chi_{\R^n\setminus E},$$
and the set~$E$ possesses a minimal interface criterion, depending on~$s$.
More precisely, in the ``weakly nonlocal regime'' in which~$s\in\left[\frac12,1\right)$,
the set~$E$ turns out to be a local minimizer for the classical
perimeter functional: in this sense, on a large scale, the
weakly nonlocal regime is indistinguishable with respect to the classical case
and, in spite of the fractional nature of equation~\eqref{ALLEN},
its limit interface behaves in a local fashion
when~$s\in\left[\frac12,1\right)$.

Conversely, in the ``genuinely nonlocal regime''
in which~$s\in\left( 0,\frac12\right)$,
the set~$E$ turns out to be a local minimizer for the nonlocal perimeter functional
which was introduced in~\cite{MR2675483}. That is, the interface
of long-range phase transitions when~$s\in\left(0, \frac12\right)$
preserves its nonlocal features at any arbitrarily large scale,
and, as a matter of fact, the scaling properties of the associated
energy functional preserve this nonlocal character as well.
Needless to say, the persistence of the nonlocal
properties at any scale and the somehow unpleasant scaling of the associated
energies provide a number of difficulties in the analysis
of long-range phase coexistence models.\medskip

In particular, symmetry properties
of the solutions of equation~\eqref{ALLEN}
have been intensively studied, also in view of a celebrated
conjecture by E. De Giorgi in the classical case, see~\cite{DG}.
This classical conjecture asks whether or not bounded and monotone
solutions of phase transitions equations are necessarily $1$D.
In the fractional framework, a positive answer
to this problem was known in dimension~$2$
(see~\cite{MR2177165} for the case~$s=\frac12$
and~\cite{SV09, cabre-TAMS, MR3035063} for the full range~$s\in(0,1)$).
Also, in dimension~$3$, a positive answer was known
only in the weakly nonlocal regimes~$s=\frac12$
and~$s\in\left(\frac12,1\right)$, see~\cite{MR2644786,MR3148114}.
See also~\cite{2016arXiv161009295S} for a very recent contribution
about symmetry results for equation~\eqref{ALLEN}
in the weakly nonlocal regime --
as a matter of fact, the lack of ``good energy estimates''
prevented the extension of the techniques of these articles
to the strongly nonlocal regime~$s\in\left(0,\frac12\right)$.
In this sense, our Theorem~\ref{MAIN} aims at overcoming
these difficulties, by relying on the very recent paper~\cite{DSVarxiv},
which has now taken into account the weakly nonlocal regime
for equation~\eqref{ALLEN}.\medskip

After this work was completed we have 
also received a preliminary version of the article \cite{cabre}, 
in which symmetry results for fractional Allen-Cahn equations 
will be obtained also in the setting of stable solutions.
\medskip

The rest of the paper is organized as follows.
In Section~\ref{SEC:1}, we recall the notion of local minimizers
and we introduce an equivalent minimization problem in an extended space:
this part is rather technical, but absolutely non-standard,
since the lack of decay of our solution and the strongly nonlocal condition~$s\in\left(0,\frac12\right)$
make the energy diverge, hence the standard extension methods
are not available in our case and we will need to introduce a suitable energy
renormalization procedure.

In Section~\ref{SEC:2}, we relate stable and minimal solutions in the one-dimensional
case, by relying also on some layer solution theory of~\cite{cabre-sire-AIHP}.

In Section~\ref{CLASS:A} we consider the profiles of the solution
at infinity and we establish their minimality and symmetry properties.

In Section~\ref{CLASS:B}, we discuss the minimization properties
under perturbation which do not overcome the limit profiles,
and in Section~\ref{CLASS:C} we recover the minimality of a solution from that
of its limit profiles. The proof of Theorem~\ref{MAIN}.
is contained in Section~\ref{CLASS:D}.

For completeness, in Section~\ref{UL810} we also provide a variant of Theorem~\ref{MAIN}
that gives minimality and symmetry results under the assumption that the limit profiles
are two-dimensional.
\medskip

It is worth to point out that the setting in Sections~\ref{SEC:1}--\ref{CLASS:C}
is very general and it applies to all the fractional powers~$s\in(0,1)$, hence
it can be seen as a useful tool to deal with a class of problems also in extended spaces,
so to recover minimal properties of the solution from some knowledge of the limit profiles.

\section{Local minimizers in $\R^n$ and extended local minimizers in~$\R^{n+1}_+$}\label{SEC:1}

Equation~\eqref{ALLEN} lies in the class of semilinear fractional equations of the type
\begin{equation}\label{ALLEN-GEN} (-\Delta)^s u=f(u).\end{equation}
{F}rom now on,
we will denote by~$f$ a bistable nonlinearity, namely,
we assume that~$f(-1)=f(1)=0$, and there exist~$\kappa > 0$
and~$c_\kappa > 0$ such that~$ f'(t)<-c_\kappa$ for
any~$t\in[-1,-1+\kappa]\cup[1-\kappa,1]$. We also assume that
\begin{equation} \label{BISTABLE}
\int_0^1 f(\sigma)\,d\sigma>0>\int_{-1}^0 f(\sigma)\,d\sigma
.\end{equation}
To ensure that the solution is sufficiently regular
in our computations, we assume that
\begin{equation}\label{QUE}
f\in C^{1,\alpha}_{\rm loc}(\R)\end{equation}
with~$\alpha\in(0,1)$
and~$\alpha>1-2s$.
We remark that, in view
of~\eqref{QUE} here and Lemma~4.4 in~\cite{cabre-sire-AIHP}, bounded
solutions of~\eqref{ALLEN-GEN} are automatically in~$C^2(\R^n)$,
with bounded second derivatives.

The prototype for such bistable nonlinearity is, of course, the case in which~$f(t)=t-t^3$.
Also,
to describe the energy framework of nonlocal phase transitions, given~$s\in(0,1)$, $v:\R^n\to\R$
and~$\omega\subset\R^n$, we consider the functional
$$ {\mathcal{F}}_\omega(v):=
\frac{c_{n,s}}{2}\iint_{Q_\omega}\frac{|v(x)-v(y)|^2}{|x-y|^{n+2s}}\,dx\,dy+\int_\omega
F(v(x))\,dx,$$
where
$$ Q_\omega:= (\omega\times\omega)\cup
(\omega\times\omega^c)\cup(\omega^c\times\omega)$$
and
$$ F(t):= -\int_{-1}^t f(\tau)\,d\tau.$$

\begin{definition}\label{D:L} We say that~$u$ is a local minimizer
if, for any~$R>0$ and any~$\varphi\in C^\infty_0(B_R)$, it holds that~$
{\mathcal{F}}_{B_R}(u)\le{\mathcal{F}}_{B_R}(u+\varphi)$.
\end{definition}

Now we describe an extended problem
and relate its local minimization to the one in Definition~\ref{D:L}
(see~\cite{caffasil}).
For this, we set~$a:=1-2s\in(-1,1)$ and~$\R^{n+1}_+:=\R^n\times(0,+\infty)$.
Then, given any~$V:\R^{n+1}_+\to\R$ and~$\Omega\subset\R^{n+1}$, we define
$$ {\mathcal{E}}_\Omega(V):=
\frac{\tilde c_{n,s}}{2}\int_{\Omega^+}z^a |\nabla V(x,z)|^2\,dx\,dz+\int_{\Omega_0}
F(V(x,0))\,dx,$$
where~$\Omega^+:=\Omega\cap\R^{n+1}_+$ and~$\Omega_0:=\Omega\cap\{z=0\}$.
Here, we used the notation~$
(x,z)\in\R^n\times(0,+\infty)$ to denote the variables of~$\R^{n+1}_+$ and~$\tilde c_{n,s}>0$
is a normalization constant.
Given~$R>0$, we also denote
\begin{eqnarray*}&&{\mathcal{B}}_R:= B_R\times(-R,R)=\big\{(x,z)\in
\R^{n}\times\R {\mbox{ s.t. }} x\in B_R {\mbox{ and }} |z|<R\big\}\\{\mbox{and }}&&
{\mathcal{B}}_R^+:= {\mathcal{B}}_R\cap\{z>0\}=
B_R\times(0,R)=\big\{(x,z)\in
\R^{n}\times\R {\mbox{ s.t. }} x\in B_R {\mbox{ and }} z\in(0,R)\big\}.\end{eqnarray*}
In this setting, we have the following notation:
\begin{definition}\label{D:L:2} We say that~$U$ is an extended local minimizer
if, for any~$R>0$ and any~$\Phi\in C^\infty_0({\mathcal{B}}_R)$, it holds that~$
{\mathcal{E}}_{{\mathcal{B}}_R}(U)\le{\mathcal{E}}_{{\mathcal{B}}_R}(U+\Phi)$.
\end{definition}

The reader can compare Definitions~\ref{D:L} and~\ref{D:L:2}. Also, given~$v\in L^\infty(\R^n)$,
we consider the $a$-harmonic extension of~$v$ to~$\R^{n+1}_+$ as the function~$E_v:
\R^{n+1}_+\to\R$
obtained by convolution with the Poisson kernel of order~$s$. More explicitly, we set
\begin{equation}\label{0djhicseodfeeyyeyeye} P(x,z):= \bar c_{n,s} \frac{z^{2s}}{(|x|^2+z^2)^{\frac{n+2s}2}}.\end{equation}
In this framework, $\bar c_{n,s}$ is a positive normalization constant such that
$$ \int_{\R^n} P(x,z)\,dx=1,$$
see e.g.~\cite{MR3461641}. Then we set
$$ E_v(x,z):=\int_{\R^n} P(x-y,z)\,v(y)\,dy=
\int_{\R^n} P(y,z)\,v(x-y)\,dy.$$
We remark that when~$v\in C^\infty_0(\R^n)$, the function~$E_v$ can also be obtained
by
minimization
of the associated Dirichlet energy, namely
\begin{equation} \label{DIRmi}
\inf_{{V\in C^\infty_0(\R^{n+1})}\atop{V(x,0)=v(x)}}
\int_{\R^{n+1}_+}z^a |\nabla V(x,z)|^2\,dx\,dz
=
\int_{\R^{n+1}_+}z^a |\nabla E_v(x,z)|^2\,dx\,dz,\end{equation}
see Lemma 4.3.3 in~\cite{bucur}. Nevertheless,
we want to consider here the more general framework in which~$v$
is bounded, but not necessarily decaying at infinity, and this will produce a number
of difficulties, also due to the lack of ``good'' functional settings.

We also remark that the setting in~\eqref{DIRmi} and the normalization
constant~$\tilde c_{n,s}$
are compatible with the choice of the constant in~\eqref{cns},
since, for any $v\in C^\infty_0(\R^n)$,
\begin{equation}\label{COMP:AK}
\frac{c_{n,s}}{2}\iint_{\R^{2n}}\frac{|v(x)-v(y)|^2}{|x-y|^{n+2s}}\,dx\,dy=
\frac{\tilde c_{n,s}}{2}\int_{\R^{n+1}_+}z^a |\nabla E_v(x,z)|^2\,dx\,dz,
\end{equation}
see e.g. formula~(4.3.15) in~\cite{bucur}.

Since these normalization constants will not play any role in the following computations,
with a slight abuse of notation, for the sake of simplicity,
we just omit them in the sequel.

In our setting, for functions~$v$ with no decay at infinity,
formula~\eqref{COMP:AK} does not make sense, since both the terms could diverge.
Nevertheless, we will be able to overcome this difficulty by an energy
renormalization procedure, based on the formal substraction of the infinite energy.
The rigorous details of this procedure are discussed in the
following\footnote{We observe that Proposition~\ref{0OAPQO182:P}
here is also related to the extension method in Lemma~7.2
in~\cite{MR2675483}, where suitable trace and extended energies
are compared in the unit ball:
in a sense, since
Proposition~\ref{0OAPQO182:P} here
compares energies defined in the whole of the space, it can be viewed
as a ``global'', or ``renormalized'', version of
Lemma~7.2
in~\cite{MR2675483}. }
result:

\begin{proposition}\label{0OAPQO182:P}
Let~$n\ge2$ and~$s\in(0,1)$.
For any~$v\in W^{2,\infty}(\R^n)$
and any~$\varphi\in C^\infty_0(\R^n)$, it
holds that
\begin{eqnarray*} +\infty&>& 
\iint_{\R^{2n}}
\frac{|(v+\varphi)(x)-(v+\varphi)(y)|^2-
|v(x)-v(y)|^2
}{|x-y|^{n+2s}}\,dx\,dy\\&=&
\lim_{R\to+\infty}
\int_{ {\mathcal{B}}_R^+ }
z^a\big(|\nabla E_{v+\varphi}(x,z)|^2
-|\nabla E_v(x,z)|^2\big)\,dx\,dz\\&
=& \lim_{R\to+\infty} \inf_{{\Phi\in
C^\infty_0({\mathcal{B}}_R)}\atop{\Phi(x,0)=\varphi(x)}}
\int_{ {\mathcal{B}}_R^+ }
z^a\big(|\nabla (E_{v}+\Phi)(x,z)|^2
-|\nabla E_v(x,z)|^2\big)\,dx\,dz
.\end{eqnarray*}
\end{proposition}

\begin{proof} For concreteness, we
suppose that the support of~$\varphi$
lies in~$B_1$.
We take~$\tau\in C^\infty_0(B_2,\,[0,1])$,
with~$\tau=1$ in~$B_1$, and we let
\begin{equation}\label{v k} \tau_k(x):=\tau\left(\frac{x}{k}\right)
\quad{\mbox{ and }}\quad
v_k(x):=\tau_k(x)\,v(x).\end{equation}
We also set~$w_k:= v-v_k$.
In this way, $w_k$ is bounded, uniformly
Lipschitz and
vanishes in~$B_k$.
In particular,
\begin{equation}\label{9:PAA}
\frac{\big|w_k(x)-w_k(y)\big|\,\big|
\varphi(x)-\varphi(y)\big|
}{|x-y|^{n+2s}}\le
\frac{C\,\min\{1,|x-y|\}\,\big|
\varphi(x)-\varphi(y)\big|
}{|x-y|^{n+2s}}\in L^1(\R^{2n}).\end{equation}
Furthermore, we have that
\begin{eqnarray*}
&& \iint_{\R^{2n}}
\frac{|(v+\varphi)(x)-(v+\varphi)(y)|^2-
|v(x)-v(y)|^2
}{|x-y|^{n+2s}}\,dx\,dy\\&&\qquad
-
\iint_{\R^{2n}}
\frac{|(v_k+\varphi)(x)-(v_k+\varphi)(y)|^2-
|v_k(x)-v_k(y)|^2
}{|x-y|^{n+2s}}\,dx\,dy
\\ &=&2
\iint_{\R^{2n}}
\frac{\big(v(x)-v(y)\big)\big(
\varphi(x)-\varphi(y)\big)
}{|x-y|^{n+2s}}\,dx\,dy\\&&\qquad
-2\iint_{\R^{2n}}
\frac{\big(v_k(x)-v_k(y)\big)\big(
\varphi(x)-\varphi(y)\big)
}{|x-y|^{n+2s}}\,dx\,dy
\\ &=&
2\iint_{\R^{2n}}
\frac{\big(w_k(x)-w_k(y)\big)\big(
\varphi(x)-\varphi(y)\big)
}{|x-y|^{n+2s}}\,dx\,dy
.\end{eqnarray*}
This, \eqref{9:PAA}
and the Dominated Convergence
Theorem give that
\begin{equation}\label{K9:00:01}
\begin{split}&
\lim_{k\to+\infty}
\iint_{\R^{2n}}
\frac{|(v_k+\varphi)(x)-(v_k+\varphi)(y)|^2-
|v_k(x)-v_k(y)|^2
}{|x-y|^{n+2s}}\,dx\,dy
\\&\qquad=
\iint_{\R^{2n}}
\frac{|(v+\varphi)(x)-(v+\varphi)(y)|^2-
|v(x)-v(y)|^2
}{|x-y|^{n+2s}}\,dx\,dy
.\end{split}\end{equation}
Also,
$$ \frac{\big|v(x)-v(y)\big|\,\big|
\varphi(x)-\varphi(y)\big|
}{|x-y|^{n+2s}}\le
\frac{C\,\min\{1,|x-y|\}\,\big|
\varphi(x)-\varphi(y)\big|
}{|x-y|^{n+2s}}\in L^1(\R^{2n}),$$
and therefore
\begin{equation}\label{K9:00:02}
\begin{split}&
\iint_{\R^{2n}}
\frac{|(v+\varphi)(x)-(v+\varphi)(y)|^2-
|v(x)-v(y)|^2
}{|x-y|^{n+2s}}\,dx\,dy\\&\qquad=
\iint_{\R^{2n}}
\frac{
\big(v(x)-v(y)\big)\,\big(
\varphi(x)-\varphi(y)\big)
}{|x-y|^{n+2s}}\,dx\,dy+
\iint_{\R^{2n}}\frac{|\varphi(x)-\varphi(y)|^2}{|x-y|^{n+2s}}\,dx\,dy
<+\infty
.\end{split}\end{equation}
On the other hand, recalling~\eqref{0djhicseodfeeyyeyeye}, 
\[
z^a |\nabla P(x,z)|\le
\frac{Cz(|x|+z)}{(|x|^2+z^2)^{\frac{n+2s+2}{2}}}
+\frac{C}{(|x|^2+z^2)^{\frac{n+2s}{2}}}
\le
\frac{C}{(|x|^2+z^2)^{\frac{n+2s}{2}}}
.\]
This implies that
\begin{eqnarray*}
&& z^a |\nabla E_\varphi(x,z)|=
\left| z^a \int_{\R^n} \nabla
P(x-y,z)\,\varphi(y)\,dy
\right|\\&&\qquad\le
C \int_{B_1} |\nabla P(x-y,z)|\,dy
\le C\int_{B_1}
\frac{dy}{(|x-y|^2+z^2)^{\frac{n+2s}{2}}}
.\end{eqnarray*}
Hence, for any~$x\in\R^n\setminus B_2$,
\begin{equation}\label{HENCE1}
z^a |\nabla E_\varphi(x,z)|
\le C\int_{B_1}
\frac{dy}{(|x|^2+z^2)^{\frac{n+2s}{2}}}
= \frac{C}{(|x|^2+z^2)^{\frac{n+2s}{2}}}
\end{equation}
and, if~$x\in B_2$,
\begin{equation}\label{HENCE2}
z^a |\nabla E_\varphi(x,z)|\le
C\int_{B_1}
\frac{dy}{(0+z^2)^{\frac{n+2s}{2}}}
\le
\frac{C}{z^{n+2s}}
.\end{equation}
We also notice that
\begin{equation}\label{6bis}\begin{split}& z^a\,|\nabla E_\varphi(x,z)|\le
z^a\,\left| \int_{B_1} P(y,z)\,\nabla\varphi(x-y)\,dy\right|
\\&\qquad\le C z^a\int_{\R^n} P(y,z)\,dy\le Cz^a\in L^1(B_2\times(0,2)).\end{split}\end{equation}
In addition, $\| E_{v_k}\|_{L^\infty(\R^{n+1}_+)}\le
\|v_k\|_{L^\infty(\R^n)}$;
as a consequence
of these observations, setting
\begin{equation}\label{elle R}
L_R:=
(\partial B_R)\times (0,R)\quad{\mbox{
and }}\quad U_R:= B_R\times\{R\}, \end{equation}we have that
\begin{equation}\label{889AO}\begin{split}&
\lim_{R\to+\infty}
\int_{L_R\cup U_R }
z^a |E_{v_k}(x,z)|\,|
\partial_\nu 
E_{\varphi}(x,z)|\,d{\mathcal{H}}^{n}(x,z)
\le\lim_{R\to+\infty}
\int_{L_R\cup U_R }
\frac{C\,d{\mathcal{H}}^{n}(x,z)}{
(|x|^2+z^2)^{\frac{n+2s}{2}}}
\\ &\qquad\qquad\le\lim_{R\to+\infty}
\int_{L_R\cup U_R }
\frac{C\,d{\mathcal{H}}^{n}(x,z)}{
R^{n+2s}} \le \lim_{R\to+\infty}
\frac{CR^n}{R^{n+2s}}=0,
\end{split}\end{equation}
where~$\partial_\nu$ denotes the
external normal derivative to the boundary
of~${\mathcal{B}}_R^+$.

Accordingly,
\begin{eqnarray*}
&& \lim_{R\to+\infty}
\int_{ {\mathcal{B}}_R^+ }
z^a\big(|\nabla E_{v_k+\varphi}(x,z)|^2
-|\nabla E_{v_k}(x,z)|^2\big)\,dx\,dz\\
&=&
\lim_{R\to+\infty}
\int_{ {\mathcal{B}}_R^+ }
z^a\big(|\nabla E_{\varphi}(x,z)|^2
+2\nabla E_\varphi(x,y)\cdot
\nabla E_{v_k}(x,z)\big)\,dx\,dz
\\ &=&
\lim_{R\to+\infty}
\int_{ {\mathcal{B}}_R^+ }
z^a |\nabla E_{\varphi}(x,z)|^2
+2{\rm div}\,\big(z^a
E_{v_k}(x,y)
\nabla E_{\varphi}(x,z)\big) \,dx\,dz\\
&=&
\lim_{R\to+\infty}
\int_{ {\mathcal{B}}_R^+ }
z^a |\nabla E_{\varphi}(x,z)|^2\,dx\,dz
+2\int_{\partial {\mathcal{B}}_R^+ }
z^a E_{v_k}(x,z)
\partial_\nu 
E_{\varphi}(x,z)\,d{\mathcal{H}}^{n}(x,z)\\
&=&
\lim_{R\to+\infty}
\int_{ {\mathcal{B}}_R^+ }
z^a |\nabla E_{\varphi}(x,z)|^2\,dx\,dz
+2\int_{B_R} v_k(x)\,(-\Delta)^s\varphi(x)\,dx
\\ &=&
\int_{ \R^{n+1}_+}
z^a |\nabla E_{\varphi}(x,z)|^2\,dx\,dz
+2\int_{\R^n} v_k(x)
\,(-\Delta)^s\varphi(x)\,dx.
\end{eqnarray*}
Similarly,
\begin{eqnarray*}
&& \lim_{R\to+\infty}
\int_{ {\mathcal{B}}_R^+ }
z^a\big(|\nabla E_{v+\varphi}(x,z)|^2
-|\nabla E_{v}(x,z)|^2\big)\,dx\,dz\\
\\&=&
\int_{ \R^{n+1}_+}
z^a |\nabla E_{\varphi}(x,z)|^2\,dx\,dz
+2\int_{\R^n} v(x)
\,(-\Delta)^s\varphi(x)\,dx.\end{eqnarray*}
Since~$|v_k (-\Delta)^s\varphi|\le|v(-\Delta)^s\varphi|\le\frac{C}{1+|x|^{n+2s}}\in L^1(\R^n)$,
we thus obtain that
\begin{equation}\label{89:0128esd}
\begin{split}
&\lim_{k\to+\infty}
\int_{\R^{n+1}_+}
z^a\big(|\nabla E_{v_k+\varphi}(x,z)|^2
-|\nabla E_{v_k}(x,z)|^2\big)\,dx\,dz
\\=\;&\lim_{k\to+\infty}\lim_{R\to+\infty}
\int_{ {\mathcal{B}}_R^+ }
z^a\big(|\nabla E_{v_k+\varphi}(x,z)|^2
-|\nabla E_{v_k}(x,z)|^2\big)\,dx\,dz
\\
=\;&\lim_{k\to+\infty}
\left[\int_{ \R^{n+1}_+}
z^a |\nabla E_{\varphi}(x,z)|^2\,dx\,dz
+2\int_{\R^n} v_k(x)
\,(-\Delta)^s\varphi(x)\,dx
\right]\\ =\;&
\int_{ \R^{n+1}_+}
z^a |\nabla E_{\varphi}(x,z)|^2\,dx\,dz
+2\int_{\R^n} v(x)
\,(-\Delta)^s\varphi(x)\,dx\\
=\;&
\lim_{R\to+\infty}
\int_{ {\mathcal{B}}_R^+ }
z^a\big(|\nabla E_{v+\varphi}(x,z)|^2
-|\nabla E_{v}(x,z)|^2\big)\,dx\,dz.
\end{split}\end{equation}
Moreover, from~\eqref{COMP:AK}, we know that
\begin{eqnarray*}
&&  
\iint_{\R^{2n}}
\frac{|(v_k+\varphi)(x)-(v_k+\varphi)(y)|^2-
|v_k(x)-v_k(y)|^2
}{|x-y|^{n+2s}}\,dx\,dy\\&=&
\lim_{R\to+\infty}
\int_{ {\mathcal{B}}_R^+ }
z^a\big(|\nabla E_{v_k+\varphi}(x,z)|^2
-|\nabla E_{v_k}(x,z)|^2\big)\,dx\,dz
.\end{eqnarray*}
Putting together this with~\eqref{K9:00:01},
\eqref{K9:00:02} and~\eqref{89:0128esd}, we
conclude that
\begin{eqnarray*} +\infty&>& 
\iint_{\R^{2n}}
\frac{|(v+\varphi)(x)-(v+\varphi)(y)|^2-
|v(x)-v(y)|^2
}{|x-y|^{n+2s}}\,dx\,dy\\&=&
\lim_{R\to+\infty}
\int_{ {\mathcal{B}}_R^+ }
z^a\big(|\nabla E_{v+\varphi}(x,z)|^2
-|\nabla E_v(x,z)|^2\big)\,dx\,dz.\end{eqnarray*}
Therefore, to complete the proof of the desired claim,
it remains to show that
\begin{equation}\label{0OAPQO182}
\begin{split}
&\lim_{R\to+\infty}
\int_{ {\mathcal{B}}_R^+ }
z^a\big(|\nabla E_{v+\varphi}(x,z)|^2
-|\nabla E_v(x,z)|^2\big)\,dx\,dz
\\
=\;& \lim_{R\to+\infty} \inf_{ {\Phi\in
C^\infty_0({\mathcal{B}}_R)}\atop{\Phi(x,0)=\varphi(x)}}
\int_{ {\mathcal{B}}_R^+ }
z^a\big(|\nabla (E_{v}+\Phi)(x,z)|^2
-|\nabla E_v(x,z)|^2\big)\,dx\,dz
.\end{split}\end{equation}
To this end,
we take~$\theta\in C^\infty_0({\mathcal{B}}_1,\,[0,1])$
with~$\theta=1$ in~${\mathcal{B}}_{1/2}$
and we set~$\theta_R(x,z):=\theta\left(\frac{x}{R},\frac{z}{R}\right)$
and~$\Phi_R:=E_\varphi \theta_R$.
We observe that
$$ |E_\varphi(x,z)|\le \int_{B_1}
\frac{C z^{2s}}{(|x-y|^2+z^2)^{\frac{n+2s}2}}\,dy$$
and therefore, if~$|x|\ge2$,
\begin{equation*} 
|E_\varphi(x,z)|\le 
\frac{C z^{2s}}{(|x|^2+z^2)^{\frac{n+2s}2}}\le
\frac{C z^{2s}}{(1+|x|^2+z^2)^{\frac{n+2s}2}}
.\end{equation*}
Similarly, if~$|x|\le2$ and~$z\ge1$,
\begin{equation*}
|E_\varphi(x,z)|\le\int_{ B_1}
\frac{C z^{2s}}{(0+z^2)^{\frac{n+2s}2}}\,dy\le
\frac{C z^{2s}}{(1+|x|^2+z^2)^{\frac{n+2s}2}},
\end{equation*}
up to renaming~$C$.
Also, if~$|x|\le2$ and~$z\in(0,2)$,
$$ |E_\varphi(x,z)|\le C\int_{\R^n}P(x-y,z)\,dy=C.
$$
In view of these estimates, we have that
\begin{equation}\label{STIME:EPHI1}
\begin{split}& \int_{\R^{n+1}_+} z^a |E_\varphi(x,z)|^2\,dx\,dz
\le
C+\int_{\R^{n+1}_+\setminus {\mathcal{B}}_2} z^a |E_\varphi(x,z)|^2\,dx\,dz\\
&\qquad\le C+\int_{\R^{n+1}_+\setminus {\mathcal{B}}_2} 
\frac{C z^{1+2s}}{(1+|x|^2+z^2)^{{n+2s}}}
\,dx\,dz\\&\qquad
\le
C+\int_{\R^{n+1}_+\setminus {\mathcal{B}}_2} 
\frac{C }{(1+|x|^2+z^2)^{{n+s-\frac12}}}
\,dx\,dz\le C,
\end{split}\end{equation}
since~$2n+2s-1>n+1$.

Similarly, recalling~\eqref{HENCE1},
\eqref{HENCE2} and~\eqref{6bis},
\begin{equation*}
\begin{split}& \int_{\R^{n+1}_+} z^a |\nabla E_\varphi(x,z)|^2\,dx\,dz
\\ \le\;& C+\int_{ \{(x,z):\, x\in B_2,\, z>1\}}
z^{-a} \big(z^a |\nabla E_\varphi(x,z)|\big)^2\,dx\,dz
+\int_{ \{(x,z): \, x\in \R^n\setminus B_2,\, z\in(0,1]\}}
z^{-a} \big(z^a |\nabla E_\varphi(x,z)|\big)^2\,dx\,dz
\\&\qquad+\int_{ \{(x,z):\, x\in \R^n\setminus B_2,\,z>1\}}
z^{-a} \big(z^a |\nabla E_\varphi(x,z)|\big)^2\,dx\,dz
\\ \le\;& C+
\int_{ \{(x,z):\, x\in B_2,\, z>1\}}
\frac{C}{z^{2n+1+2s}}\,dx\,dz
+\int_{ \{(x,z):\, x\in \R^n\setminus B_2,\, z\in(0,1]\}}
\frac{C}{z^{1-2s} (|x|^2+0)^{n+2s}}\,dx\,dz\\&\qquad
+\int_{ \{(x,z):\, x\in \R^n\setminus B_2,\,z>1\}}
\frac{C}{(|x|^2+z^2)^{n+2s}}\,dx\,dz
\\ \le\;&C,
\end{split}
\end{equation*}
and therefore
\begin{equation}\label{STIME:EPHI2}
\int_{\R^{n+1}_+\setminus{\mathcal{B}}_{R/2}} z^a |\nabla E_\varphi(x,z)|^2\,dx\,dz\le\delta(R),
\end{equation}
with~$\delta(R)$ infinitesimal as~$R\to+\infty$.

In addition,
\begin{eqnarray*} \Big| |\nabla \Phi_R|^2-|\nabla E_\varphi|^2\Big|
&\le& E_\varphi^2|\nabla\theta_R|^2
+|\nabla E_\varphi|^2 |1-\theta_R^2|
+2|E_\varphi| |\theta_R|
|\nabla E_\varphi||\nabla\theta_R|\\
&\le& \frac{C E_\varphi^2}{R^2 }+
C|\nabla E_\varphi|^2 \chi_{\R^{n+1}_+\setminus {\mathcal{B}}_{R/2}}.
\end{eqnarray*}
Therefore, in light of~\eqref{STIME:EPHI1} and~\eqref{STIME:EPHI2},
\begin{eqnarray*}&& \left|
\int_{\R^{n+1}_+} z^a |\nabla \Phi_R(x,z)|^2\,dx\,dz
-
\int_{\R^{n+1}_+} z^a |\nabla E_\varphi(x,z)|^2\,dx\,dz\right|\\ &\le&
\frac{C}{R^2}\int_{\R^{n+1}_+} z^a
|E_\varphi(x,z)|^2\,dx\,dz+
C\int_{\R^{n+1}_+\setminus {\mathcal{B}}_{R/2}} z^a
|\nabla E_\varphi(x,z)|^2 \,dx\,dz \\&\le&
\frac{C}{R^2}+\delta(R).
\end{eqnarray*}
Consequently, we find that
\begin{equation}\label{PAal203948:PP}
\begin{split}
& \lim_{R\to+\infty} \inf_{ {\Phi\in
C^\infty_0({\mathcal{B}}_R)}\atop{\Phi(x,0)=\varphi(x)}}
\int_{ {\mathcal{B}}_R^+ }
z^a\big(|\nabla (E_{v}+\Phi)(x,z)|^2
-|\nabla E_v(x,z)|^2\big)\,dx\,dz
\\ \le\;&
\lim_{R\to+\infty}
\int_{ {\mathcal{B}}_R^+ }
z^a\big(|\nabla (E_{v}+\Phi_R)(x,z)|^2
-|\nabla E_v(x,z)|^2\big)\,dx\,dz
\\
=\;&
\lim_{R\to+\infty}
\int_{ {\mathcal{B}}_R^+ }
z^a\big(|\nabla \Phi_R(x,z)|^2 +2\nabla E_v(x,z)\cdot\nabla\Phi_R(x,z)\big)\,dx\,dz
\\ 
\le\;&
\lim_{R\to+\infty} \left[ \frac{C}{R^2}+\delta(R)+
\int_{ {\mathcal{B}}_R^+ }
z^a\big(|\nabla E_\varphi(x,z)|^2 +2\nabla E_v(x,z)\cdot\nabla
E_\varphi(x,z)\big)\,dx\,dz\right.\\&\qquad
\left.+2
\int_{ {\mathcal{B}}_R^+ }
z^a\nabla E_v(x,z)\cdot\nabla\big(\Phi_R(x,z)-E_\varphi(x,z)\big)\,dx\,dz\right]\\
\le\;&
\lim_{R\to+\infty} \left[ \frac{C}{R^2}+\delta(R)+
\int_{ {\mathcal{B}}_R^+ }
z^a\big(|\nabla E_\varphi(x,z)|^2 +2\nabla E_v(x,z)\cdot\nabla
E_\varphi(x,z)\big)\,dx\,dz\right.\\&\qquad\left.
+2\sup_{ {\Phi\in
C^\infty_0({\mathcal{B}}_R)}\atop{\Phi(x,0)=\varphi(x)}}
\int_{ {\mathcal{B}}_R^+ }
z^a\nabla E_v(x,z)\cdot\nabla\big(\Phi(x,z)-E_\varphi(x,z)\big)\,dx\,dz\right]
.\end{split}\end{equation}
Now we claim that
\begin{equation}\label{PAal203948}
\lim_{R\to+\infty} \sup_{ {\Phi\in
C^\infty_0({\mathcal{B}}_R)}\atop{\Phi(x,0)=\varphi(x)}}
\left|\int_{ {\mathcal{B}}_R^+ }
z^a\nabla E_v(x,z)\cdot\nabla\big(\Phi(x,z)-E_\varphi(x,z)\big)\,dx\,dz\right|
=0.
\end{equation}
To check this, we recall the notation in~\eqref{elle R}
and observe that
\begin{equation}\label{012owe2eudyf8i:00}
\begin{split}
&\lim_{R\to+\infty}\sup_{ {\Phi\in
C^\infty_0({\mathcal{B}}_R)}\atop{\Phi(x,0)=\varphi(x)}}
\left|\int_{ {\mathcal{B}}_R^+ }
z^a\nabla E_v(x,z)\cdot\nabla\big(\Phi(x,z)-E_\varphi(x,z)\big)\,dx\,dz\right|\\
=\,&
\lim_{R\to+\infty}\sup_{ {\Phi\in
C^\infty_0({\mathcal{B}}_R)}\atop{\Phi(x,0)=\varphi(x)}}\left|
\int_{ {\mathcal{B}}_R^+ } {\rm div}\,\Big(
z^a \Phi(x,z) \nabla E_v(x,z)\Big)\,dx\,dz
-\int_{ {\mathcal{B}}_R^+ } {\rm div}\,\Big(
z^a E_v(x,z)\nabla E_\varphi(x,z)\Big)\,dx\,dz\right|\\
=\,&
\lim_{R\to+\infty}\left| \int_{B_R} \varphi(x)\, (-\Delta)^s v(x)\,dx
- \int_{B_R} v(x)\, (-\Delta)^s \varphi(x)\,dx\right.\\
&\quad\qquad\left.-\int_{L_R\cup U_R} 
z^a E_v(x,z)\partial_\nu E_\varphi(x,z)\,d{\mathcal{H}}^n(x,z)\right|
.\end{split}
\end{equation}
Also, as in~\eqref{889AO}, we have that
\begin{equation*}
\lim_{R\to+\infty}
\int_{L_R\cup U_R }
z^a |E_{v}(x,z)|\,|
\partial_\nu 
E_{\varphi}(x,z)|\,d{\mathcal{H}}^{n}(x,z)=0.
\end{equation*}
Hence, fixing~$k$ and
recalling the notation in~\eqref{v k}, we derive from~\eqref{012owe2eudyf8i:00}
that
\begin{eqnarray*}
&&\lim_{R\to+\infty}\sup_{ {\Phi\in
C^\infty_0({\mathcal{B}}_R)}\atop{\Phi(x,0)=\varphi(x)}}\left|
\int_{ {\mathcal{B}}_R^+ }
z^a\nabla E_v(x,z)\cdot\nabla\big(\Phi(x,z)-E_\varphi(x,z)\big)\,dx\,dz\right| \\
&=& \left|\int_{\R^n} \varphi(x)\, (-\Delta)^s v(x)\,dx
- \int_{\R^n} v(x)\, (-\Delta)^s \varphi(x)\,dx\right| \\
&=& \left|\int_{\R^n} \varphi(x)\, (-\Delta)^s v_k(x)\,dx
- \int_{\R^n} v_k(x)\, (-\Delta)^s \varphi(x)\,dx
\right.\\
&&\qquad+\left. \int_{\R^n} \varphi(x)\, (-\Delta)^s w_k(x)\,dx
- \int_{\R^n} w_k(x)\, (-\Delta)^s \varphi(x)\,dx\right|
\\ &=&\left|
\int_{\R^n} \varphi(x)\, (-\Delta)^s w_k(x)\,dx
- \int_{\R^n} w_k(x)\, (-\Delta)^s \varphi(x)\,dx\right|.
\end{eqnarray*}
Since~$|(-\Delta)^s \varphi(x)|\le\frac{C}{1+|x|^{n+2s}}\in L^1(\R^n)$,
we can take the limit as~$k\to+\infty$ and use the Dominated Convergence Theorem,
to obtain that
\begin{equation}\label{udofjvw9etgierufgh}
\begin{split}
&\lim_{R\to+\infty}\sup_{ {\Phi\in
C^\infty_0({\mathcal{B}}_R)}\atop{\Phi(x,0)=\varphi(x)}}
\left|\int_{ {\mathcal{B}}_R^+ }
z^a\nabla E_v(x,z)\cdot\nabla\big(\Phi(x,z)-E_\varphi(x,z)\big)\,dx\,dz\right|\\
=\;&\lim_{k\to+\infty}\left| \int_{\R^n} \varphi(x)\, (-\Delta)^s w_k(x)\,dx
- \int_{\R^n} w_k(x)\, (-\Delta)^s \varphi(x)\,dx\right|\\
=\;&\lim_{k\to+\infty}\left| \int_{\R^n} \varphi(x)\, (-\Delta)^s w_k(x)\,dx\right|
\\ =\;&
\lim_{k\to+\infty}\left| \int_{\R^n} \varphi(x)\, (-\Delta)^s v(x)\,dx
-\int_{\R^n} \varphi(x)\, (-\Delta)^s (\tau_k v)(x)\,dx\right|
\\ =\;& \left|
\int_{\R^n} \varphi(x)\, (-\Delta)^s v(x)\,dx
-
\lim_{k\to+\infty}\left[ \int_{\R^n} \tau_k(x)\,\varphi(x)\, (-\Delta)^s v(x)\,dx\right.\right.\\ &
\qquad\left.\left.+
\int_{\R^n} v(x)\,\varphi(x)\, (-\Delta)^s \tau_k(x)\,dx+
\int_{\R^n} \varphi(x)\, B(\tau_k, v)(x)\,dx\right]\right|
\\ =\;&\lim_{k\to+\infty}\left|\frac{1}{k^{2s}}
\int_{B_1} v(x)\,\varphi(x)\, (-\Delta)^s \tau\left(\frac{x}{k}\right)\,dx+
\int_{\R^n} \varphi(x)\, B(\tau_k, v)(x)\,dx\right|\\
\le\;& \lim_{k\to+\infty}\left[ \frac{C}{k^{2s}}
\int_{B_1} |v(x)|\,|\varphi(x)|\,dx+
\int_{\R^n} |\varphi(x)|\, |B(\tau_k, v)(x)|\,dx\right]
\\ =\;&\lim_{k\to+\infty}\int_{B_1} |\varphi(x)|\, |B(\tau_k, v)(x)|\,dx,
\end{split}\end{equation}
where
$$ B(f,g):= c\, \int_{\R^n} \frac{(f(x)-f(y))(g(x)-g(y))}{|x-y|^{n+2s}}\,dy,$$
for some~$c>0$, see e.g. page~636 in~\cite{MR3211862}
for such bilinear form.

Notice now that
\begin{eqnarray*}
|B(\tau_k, v)(x)|&\le&
C\int_{\R^n} \frac{\min\{ 1,\frac{|x-y|}{k}\}\,
\min\{ 1,|x-y|\}}{|x-y|^{n+2s}}\,dy\\
&\le&
\frac{C}{k} \int_{B_1(x)} \frac{|x-y|^2}{|x-y|^{n+2s}}\,dy
+
\frac{C}{k}\int_{B_k(x)\setminus B_1(x)} \frac{|x-y|}{|x-y|^{n+2s}}\,dy
+
C\int_{\R^n\setminus B_k(x)} \frac{dy}{|x-y|^{n+2s}}\\ &\le&\frac{C}{k}+\frac{C}{k^{2s}}.
\end{eqnarray*}
We plug this information into~\eqref{udofjvw9etgierufgh} and we obtain~\eqref{PAal203948},
as desired.

Now, we insert~\eqref{PAal203948} into~\eqref{PAal203948:PP} and we conclude that
\begin{eqnarray*}&& \lim_{R\to+\infty} \inf_{ {\Phi\in
C^\infty_0({\mathcal{B}}_R)}\atop{\Phi(x,0)=\varphi(x)}}
\int_{ {\mathcal{B}}_R^+ }
z^a\big(|\nabla (E_{v}+\Phi)(x,z)|^2
-|\nabla E_v(x,z)|^2\big)\,dx\,dz
\\ &&\qquad\le \lim_{R\to+\infty}
\int_{ {\mathcal{B}}_R^+ }
z^a\big(|\nabla E_\varphi(x,z)|^2 +2\nabla E_v(x,z)\cdot\nabla
E_\varphi(x,z)\big)\,dx\,dz\\
&&\qquad=
\lim_{R\to+\infty}
\int_{ {\mathcal{B}}_R^+ }
z^a\big(|\nabla E_{v+\varphi}(x,z)|^2 -|\nabla E_v(x,z)|^2\big)\,dx\,dz.
\end{eqnarray*}
Hence, to prove~\eqref{0OAPQO182}, it remains to show that
\begin{equation}\label{0OAPQO182:BIS}
\begin{split}
&\lim_{R\to+\infty}
\int_{ {\mathcal{B}}_R^+ }
z^a\big(|\nabla E_{v+\varphi}(x,z)|^2
-|\nabla E_v(x,z)|^2\big)\,dx\,dz
\\
\le \;& \lim_{R\to+\infty} \inf_{ {\Phi\in
C^\infty_0({\mathcal{B}}_R)}\atop{\Phi(x,0)=\varphi(x)}}
\int_{ {\mathcal{B}}_R^+ }
z^a\big(|\nabla (E_{v}+\Phi)(x,z)|^2
-|\nabla E_v(x,z)|^2\big)\,dx\,dz
.\end{split}\end{equation}
For this, we fix~$\Psi\in
C^\infty_0({\mathcal{B}}_R)$ with~$\Psi(x,0)=\varphi(x)$ and we use again~\eqref{PAal203948}
to see that
\begin{equation} \label{00:10023a:023948}
\begin{split}
& \int_{ {\mathcal{B}}_R^+ }
z^a\big(|\nabla (E_{v}+\Psi)(x,z)|^2
-|\nabla E_v(x,z)|^2\big)\,dx\,dz\\
=\;&
\int_{ {\mathcal{B}}_R^+ }
z^a\big(|\nabla \Psi(x,z)|^2+2z^a\nabla\Psi(x,z)\cdot\nabla E_v(x,z)\big)\,dx\,dz
\\
=\;&
\int_{ {\mathcal{B}}_R^+ }
z^a\big(|\nabla \Psi(x,z)|^2+2z^a\nabla E_\varphi(x,z)\cdot\nabla E_v(x,z)\big)\,dx\,dz
\\&\qquad+2
\int_{ {\mathcal{B}}_R^+ }
z^a\nabla E_v(x,z)\cdot\nabla\big(\Psi(x,z)-E_\varphi(x,z)\big)\,dx\,dz
\\ \ge\;&
\int_{ {\mathcal{B}}_R^+ }
z^a\big(|\nabla \Psi(x,z)|^2+2z^a\nabla E_\varphi(x,z)\cdot\nabla E_v(x,z)\big)\,dx\,dz
\\&\qquad-2\sup_{ {\Phi\in
C^\infty_0({\mathcal{B}}_R)}\atop{\Phi(x,0)=\varphi(x)}}
\left|\int_{ {\mathcal{B}}_R^+ }
z^a\nabla E_v(x,z)\cdot\nabla\big(\Phi(x,z)-E_\varphi(x,z)\big)\,dx\,dz\right|
\\ =\;&
\int_{ {\mathcal{B}}_R^+ }
z^a\big(|\nabla \Psi(x,z)|^2+2z^a\nabla E_\varphi(x,z)\cdot\nabla E_v(x,z)\big)\,dx\,dz
-\mu(R),
\end{split}\end{equation}
with~$\mu(R)$ independent of~$\Psi$ and infinitesimal as~$R\to+\infty$.

Now we take~$\Psi_*\in C^\infty_0({\mathcal{B}}_R)$ such that~$\Psi_*(x,0)=\varphi(x)$
that minimizes the functional~$\Psi\mapsto\int_{ {\mathcal{B}}_R^+ }z^a|\nabla\Psi(x,z)|^2\,dx\,dz$
in such class. Then, using the variational equation for minimizers
inside~${\mathcal{B}}_R^+$ and the fact that~$\Psi_*(x,0)=E_\varphi(x,0)$, we see that
\begin{eqnarray*}
&& \int_{ {\mathcal{B}}_R^+ }
z^a\big(|\nabla \Psi(x,z)|^2-|\nabla E_\varphi(x,z)|^2\big)\,dx\,dz\\
&\ge&
\int_{ {\mathcal{B}}_R^+ }
z^a\big(|\nabla \Psi_*(x,z)|^2-|\nabla E_\varphi(x,z)|^2\big)\,dx\,dz\\
&=&\int_{ {\mathcal{B}}_R^+ }
z^a\nabla\big(\Psi_*-E_\varphi\big)(x,z)\cdot
\nabla\big(\Psi_*+E_\varphi\big)(x,z)
\,dx\,dz\\
&=& \int_{ {\mathcal{B}}_R^+ }{\rm div}\,\Big(
z^a\big(\Psi_*-E_\varphi\big)(x,z)
\nabla\big(\Psi_*+E_\varphi\big)(x,z)\Big)
\,dx\,dz\\
&=& \int_{ L_R\cup U_R}
z^a\big(\Psi_*-E_\varphi\big)(x,z)
\partial_\nu\big(\Psi_*+E_\varphi\big)(x,z)\,d{\mathcal{H}}^n(x,z)\\
&=& -\int_{ L_R\cup U_R}
z^a E_\varphi(x,z)
\partial_\nu E_\varphi(x,z)\,d{\mathcal{H}}^n(x,z)
.
\end{eqnarray*}
Therefore, recalling~\eqref{HENCE1} and~\eqref{HENCE2}, we conclude that
\[
\int_{ {\mathcal{B}}_R^+ }
z^a\big(|\nabla \Psi(x,z)|^2-|\nabla E_\varphi(x,z)|^2\big)\,dx\,dz\ge -\int_{ L_R\cup U_R}
\frac{C}{R^{n+2s}}
\,d{\mathcal{H}}^n(x,z)\ge- \frac{C}{R^{2s}}.
\]
{F}rom this and~\eqref{00:10023a:023948} we thus obtain that
\begin{eqnarray*}
&& \int_{ {\mathcal{B}}_R^+ }
z^a\big(|\nabla (E_{v}+\Psi)(x,z)|^2
-|\nabla E_v(x,z)|^2\big)\,dx\,dz\\
&\ge& 
\int_{ {\mathcal{B}}_R^+ }
z^a\big(|\nabla E_\varphi(x,z)|^2+2z^a\nabla E_\varphi(x,z)\cdot\nabla E_v(x,z)\big)\,dx\,dz
-\mu(R)- \frac{C}{R^{2s}}\\
&=&
\int_{ {\mathcal{B}}_R^+ }
z^a\big(|\nabla E_{\varphi+v}(x,z)|^2
-|\nabla E_{v}(x,z)|^2
\big)\,dx\,dz-\mu(R)- \frac{C}{R^{2s}}.
\end{eqnarray*}
Since this is valid for any~$\Psi\in C^\infty_0({\mathcal{B}}_R)$ with~$\Psi(x,0)=\varphi(x)$,
we conclude that
\begin{eqnarray*}
&&\inf_{ {\Phi\in
C^\infty_0({\mathcal{B}}_R)}\atop{\Phi(x,0)=\varphi(x)}}
\int_{ {\mathcal{B}}_R^+ }
z^a\big(|\nabla (E_{v}+\Phi)(x,z)|^2
-|\nabla E_v(x,z)|^2\big)\,dx\,dz\\
&\ge&
\int_{ {\mathcal{B}}_R^+ }
z^a\big(|\nabla E_{v+\varphi}(x,z)|^2
-|\nabla E_v(x,z)|^2\big)\,dx\,dz-\mu(R)- \frac{C}{R^{2s}}
.\end{eqnarray*}
By taking the limit as~$R\to+\infty$,
we obtain~\eqref{0OAPQO182:BIS}, as desired,
and so we have
completed the proof of Proposition~\ref{0OAPQO182:P}.
\end{proof}

In view of Proposition~\ref{0OAPQO182:P}, 
we can now
relate the original and the extended energy functionals,
according to the following result:

\begin{corollary}\label{CPOR:EQ}
Let~$n\ge2$ and~$s\in(0,1)$.
For any~$v\in W^{2,\infty}(\R^n)$ it holds that
\begin{equation}\label{0o0oP}
\inf_{{R>0}\atop{\Phi\in C^\infty_0(\mathcal{B}_R)}}
{\mathcal{E}}_{\mathcal{B}_R}(E_v+\Phi)-{\mathcal{E}}_{\mathcal{B}_R}(E_v)\;
=\;
\inf_{{R>0}\atop{\varphi\in C^\infty_0(B_R)}}
{\mathcal{F}}_{B_R}(v+\varphi)-
{\mathcal{F}}_{B_R}(v).
\end{equation}
\end{corollary}

\begin{proof}
 Notice that, given~$\varphi\in C^\infty_0(\R^n)$,
\begin{equation}\label{IDEPHI}
\begin{split}
&\inf_{ {{R>0}\atop{\Phi\in C^\infty_0(\mathcal{B}_R)}}\atop{\Phi(x,0)=\varphi(x)}}
{\mathcal{E}}_{\mathcal{B}_R}(E_v+\Phi)-{\mathcal{E}}_{\mathcal{B}_R}(E_v)\;
\\=\;& 
\inf_{ {{R>0}\atop{\Phi\in C^\infty_0(\mathcal{B}_R)}}\atop{\Phi(x,0)=\varphi(x)}}
\int_{\R^{n+1}_+} z^a\Big( |\nabla (E_v+\Phi)(x,z)|^2-|\nabla E_v(x,z)|^2
\Big)\,dx\,dz\\ &\qquad\qquad\qquad+
\int_{\R^n}\Big( F(v(x)+\varphi(x))-F(v(x))\Big)\,dx
\\ =\;&\lim_{R\to+\infty}
\inf_{ {{\Phi\in C^\infty_0(\mathcal{B}_R)}}\atop{\Phi(x,0)=\varphi(x)}}
\int_{\R^{n+1}_+} z^a\Big( |\nabla (E_v+\Phi)(x,z)|^2-|\nabla E_v(x,z)|^2\Big)
\,dx\,dz\\&\qquad\qquad\qquad+
\int_{\R^n}\Big( F(v(x)+\varphi(x))-F(v(x))\Big)\,dx
\\=\;&
\iint_{\R^{2n}}
\frac{|(v+\varphi)(x)-(v+\varphi)(y)|^2-
|v(x)-v(y)|^2
}{|x-y|^{n+2s}}\,dx\,dy\\&\qquad\qquad\qquad
+
\int_{\R^n}\Big( F(v(x)+\varphi(x))-F(v(x))\Big)\,dx\\
=\;&
{\mathcal{F}}_{B_R}(v+\varphi)-
{\mathcal{F}}_{B_R}(v),
\end{split}\end{equation}
thanks to Proposition~\ref{0OAPQO182:P}.
Since the identity in~\eqref{IDEPHI}
is valid for any~$\varphi\in C^\infty_0(\R^n)$
(i.e. for any~$R>0$ and any~$\varphi\in C^\infty_0(B_R)$), taking the infimum in such
class we obtain~\eqref{0o0oP}, as desired.
\end{proof}

Due to Corollary~\ref{CPOR:EQ},
the following equivalence result for minimizers holds true:

\begin{proposition}\label{EQUI DL}
$E_v$ is an extended local minimizer according to Definition~\ref{D:L:2}
if and only if~$v$ is a local minimizer according to Definition~\ref{D:L}.
\end{proposition}

\begin{proof} We observe that~$E_v$ is an extended local minimizer according to Definition~\ref{D:L:2}
if and only if the first term in~\eqref{0o0oP} is nonnegative;
on the other hand, $v$ is a local minimizer according to Definition~\ref{D:L}
if and only if the last term in~\eqref{0o0oP} is nonnegative; since
the two terms in~\eqref{0o0oP} are equal, the desired result is established.
\end{proof}

In view of Proposition~\ref{EQUI DL} (see also Lemma 6.1 in~\cite{cabre-TAMS}),
it is natural to say that~$u$ is a stable solution of~\eqref{ALLEN-GEN}
if the second derivative of the associated energy functional is nonnegative,
according to the following setting:

\begin{definition}
Let~$u$ be a solution of~\eqref{ALLEN-GEN} in~$\R^n$. We say that~$u$ is stable if
$$ \int_{\R^{n+1}_+} z^a |\nabla \zeta(x,z)|^2\,dx\,dz+\int_{\R^n} F''(u(x))\,\zeta^2(x,0)\,dx\ge0$$
for any~$\zeta\in C^\infty_0(\R^{n+1})$.
\end{definition}

\section{Variational classification of $1$D solutions}\label{SEC:2}

The goal of this section is to establish the following result:

\begin{lemma}\label{CL:MO}
Let~$s\in(0,1)$ and
$v\in C^2(\R,[-1,1])$ be a stable solution of~$(-\Delta)^s v=f(v)$ in~$\R$.
Assume also that~$\dot v\ge0$. Then~$v$ is a local minimizer.
\end{lemma}

\begin{proof} The monotonicity of~$v$ implies that the following limits exist:
$$ \underline\ell:= \lim_{t\to-\infty}v(t)\le\lim_{t\to+\infty}v(t) =:\overline\ell.$$
We also consider the sequence of functions~$v_k(t):=v(t+k)$.
By the Theorem of Ascoli, up to a subsequence we know that~$v_k$ converges to~$\overline\ell$
in~$C^2_{\rm loc}(\R)$, and so, passing the equation to the limit, we conclude that~$f(\overline\ell)=0$.
Similarly, one sees that~$f(\underline\ell)=0$.
As a consequence,
\begin{equation}\label{889:p}
\underline\ell,\,\overline\ell\in\{-1,0,1\}.\end{equation}
Now, we claim that
\begin{equation}\label{NOT Z}
{\mbox{$v$ is not identically zero.}}
\end{equation}
The proof is by contradiction: if
$v$ is identically zero, we take~$\psi\in C^\infty_0(\R^2)$
and, for~$\eps>0$, we let~$\psi_\eps(x,z):=\psi(\eps x,\eps z)$. 
The stability inequality for~$\psi_\eps$ gives that
\begin{eqnarray*}
0 &\le& \int_{\R^2_+} z^a |\nabla\psi_\eps(x,z)|^2\,dx\,dz+
\int_{\R} F''(v(x))\,\psi_\eps^2(x,0)\,dx\\ &=&
\eps^2 \int_{\R^2_+} z^a |\nabla\psi(\eps x,\eps z)|^2\,dx\,dz+
\int_{\R} F''(v(x))\,\psi^2(\eps x,0)\,dx
\\ &=&
\eps^{-a} \int_{\R^2_+} Z^a |\nabla\psi(X,Z)|^2\,dX\,dZ-\eps^{-1}
\int_{\R} f'(0)\,\psi^2(X,0)\,dX
\\ &=& C_1\e^{2s-1}-C_2\e^{-1} f'(0),
\end{eqnarray*}
for some~$C_1$, $C_2>0$. {F}rom this, one obtains that
$$ f'(0)\le \lim_{\eps\searrow0} \frac{C_1\eps^{2s}}{C_2}=0.$$
This is a contradiction, since~$f'(0)>0$ and thus~\eqref{NOT Z} is proved.

To complete the proof of Lemma~\ref{CL:MO},
we now distinguish two cases, either~$v$ is constant or not.
If~$v$ is constant, then it is either identically~$-1$ or identically~$1$,
due to~\eqref{NOT Z}, and this implies the desired result.

So, we can now focus on the case in which $v$ is not constant.
Then, $\underline\ell<\overline\ell$. So, from~\eqref{889:p},
we have that~$v$ is a transition layer connecting:
\begin{enumerate}
\item either $-1$ to~$0$,
\item or $0$ to $1$,
\item or $-1$ to $1$.
\end{enumerate}
In view of Theorem 2.2(i) in~\cite{cabre-sire-AIHP}, the first two cases cannot occur
and therefore
\begin{equation}\label{78s:10234}
\lim_{t\to\pm\infty} v(t)=\pm1.
\end{equation}
Since the proof of this fact relies on the theory of layer solutions,
we provide the details of the argument that we used.
We argue for a contradiction and we suppose that
\begin{eqnarray}
&& \label{C:caso1}
{\mbox{ either }} \lim_{t\to-\infty} v(t)=-1 {\mbox{ and }} \lim_{t\to+\infty} v(t)=0\\&&\label{C:caso2}
{\mbox{ or }} \lim_{t\to-\infty} v(t)=0 {\mbox{ and }} \lim_{t\to+\infty} v(t)=1.
\end{eqnarray}
By maximum principle, we know that~$\dot v>0$.
Then, if we set either~$\tilde v(t):= 2v(t)+1$ (if the case in~\eqref{C:caso1} holds true)
or~$\tilde v(t):= 2v(t)-1$ (if~\eqref{C:caso2} holds true), we have that
the derivative of~$\tilde v$ is strictly positive and
\begin{equation}\label{78s:10234:BIS} \lim_{t\to\pm\infty}\tilde v(t)=\pm1.\end{equation}
In addition,
$$ (-\Delta)^s \tilde v(t)= 2(-\Delta)^s v(t)= 2f(v(t))=2 
f\left( \frac{\tilde v(t)\mp1}{2}\right)=
-G'(\tilde v(t)),$$
with
$$ G(r):= -2\int_{0}^r f\left( \frac{\tau\mp1}{2}\right)\,d\tau
=
-4\int_{\mp1/2}^{(r\mp1)/2} f(\sigma)\,d\sigma.$$
This and~\eqref{78s:10234:BIS} give that we are in the setting of
Theorem 2.2(i) in~\cite{cabre-sire-AIHP}. In particular, from formula~(2.8)
in~\cite{cabre-sire-AIHP} we know that
$$ 0=G(-1)-G(1)=
4\int_{\mp1/2}^{(1\mp1)/2} f(\sigma)\,d\sigma
-4\int_{\mp1/2}^{(-1\mp1)/2} f(\sigma)\,d\sigma=
4\int_{(-1\mp1)/2}^{(1\mp1)/2} f(\sigma)\,d\sigma
,$$
hence
$$ {\mbox{ either }}\quad
\int_{-1}^{0} f(\sigma)\,d\sigma=0 \quad{\mbox{ or }}\quad
\int_{0}^{1} f(\sigma)\,d\sigma=0.$$
This is in contradiction with~\eqref{BISTABLE}
and so it proves~\eqref{78s:10234}.

Hence, necessarily~$v$ is a transition layer connecting~$-1$ to~$1$
and so it is minimal due to the sliding method (see e.g. the proof of Lemma~9.1
in~\cite{VSS}).
\end{proof}

\section{Classification of the profiles at infinity}\label{CLASS:A}

In this section, we consider the two profiles 
of a given solution at infinity. Namely, if~$s\in(0,1)$ and~$u\in C^2(\R^n,[-1,1])$
is a solution of~$(-\Delta)^s u=f(u)$ in~$\R^n$, 
with~$\partial_{x_n}u>0$ in~$\R^n$,
we set
$$ \underline u(x'):=\lim_{x_n\to-\infty} u(x',x_n)\quad{\mbox{ and }}\quad
\overline u(x'):=\lim_{x_n\to+\infty} u(x',x_n).$$
In this setting, we have:

\begin{lemma} \label{02738}
Assume that~$n=3$. Then,
both~$\underline u$ and~$\overline u$ are $1$D and local minimizers.
\end{lemma}

\begin{proof} By passing the equation to the limit, we have that
\begin{equation}\label{sta000}
{\mbox{both~$\underline u$ and~$\overline u$ are 
stable solutions in~$\R^{n-1}=\R^2$.}}\end{equation}
The proof of~\eqref{sta000} is based on a general
argument (see e.g.~\cite{MR2483642}), given in details
here for the sake of completeness.
Let~$\xi\in C^\infty_0(\R^{n})$ and~$\eta\in C^\infty_0((-1,1))$, with~$\eta(0)=1$.
Given~$\e>0$, we set~$\xi_\e(x,z)=\xi_\e(x',x_n,z):=
\xi(x',z)\eta(\e x_n)$. Notice that~$\xi_\e\in C^\infty_0(\R^{n+1})$,
therefore by the stability of~$u$ and the translation invariance we
have that
\begin{eqnarray*}
0 &\le& \lim_{t\to+\infty}
C \int_{\R^{n+1}_+} z^a |\nabla \xi_\e(x,z)|^2\,dx\,dz+
\int_{\R^n} F''( u(x',x_n+t))\,\xi_\e(x,0)\,dx\\&=&
C \int_{\R^{n+1}_+} z^a |\nabla \xi_\e(x,z)|^2\,dx\,dz+
\int_{\R^n} F''( \overline u(x'))\,\xi_\e^2(x,0)\,dx\\
&=&
C \int_{\R^{n+1}_+} z^a \Big(
|\nabla \xi(x',z)|^2 \eta^2(\e x_n)+
\e^2 |\nabla \eta(\e x_n)|^2 \xi^2(x',z)+2\e
\eta(\e x_n)\xi(x',z)\nabla\eta(\e x_n)\cdot\nabla \xi(x',z)
\Big)\,dx\,dz\\&&\quad+
\int_{\R^n} F''( \overline u(x'))\,\xi^2(x',0)\eta^2(\e x_n)\,dx.
\end{eqnarray*}
Hence, taking the limit as~$\e\to0$,
$$ 0\le
C \int_{\R^{n+1}_+} z^a 
|\nabla \xi(x',z)|^2 \,dx\,dz+
\int_{\R^n} F''( \overline u(x'))\,\xi^2(x',0)\,dx,$$
and so~$\overline u$ is stable in~$\R^{n-1}$.
This proves~\eqref{sta000} for $\overline u$ (the case of~$\underline u$
is similar).

As a consequence of \eqref{sta000} and of the classification results in the plane
(see in particular Theorem 2.12 in~\cite{cabre-TAMS}, or~\cite{SV09}), we conclude that~$
\underline u$ and~$\overline u$ are~$1$D and monotone.
Then, the local minimality is a consequence of Lemma~\ref{CL:MO}.
\end{proof}

\section{Local minimization by range constraint}\label{CLASS:B}

In this section, we point out that perturbations which do not pointwise exceed
the limit profiles necessarily increase the energy. For the classical case,
this property has been exploited in Theorem~4.5 of~\cite{AAC},
Theorem~10.4 of~\cite{danielli} and 
Lemma~2.2 of~\cite{FV11}. In the framework of this paper, the result that we need is the following:

\begin{lemma}\label{INTRAP}
Let~$s\in(0,1)$ and~$u\in C^2(\R^n,[-1,1])$
be a solution of~$(-\Delta)^s u=f(u)$ in~$\R^n$, 
with~$\partial_{x_n}u>0$ in~$\R^n$.
Let
$$ \underline u(x'):=\lim_{x_n\to-\infty} u(x',x_n)\quad{\mbox{ and }}\quad
\overline u(x'):=\lim_{x_n\to+\infty} u(x',x_n).$$
Let also~$R>0$ and~$\Phi\in C^\infty_0({\mathcal{B}}_R)$ and suppose that
$$ (E_u+\Phi)(X)\in \big[ E_{\underline u}(X),\,E_{\overline u}(X)\big]\qquad{\mbox{for 
any }}X\in\R^{n+1}_+.$$
Then~${\mathcal{E}}_{ {\mathcal{B}}_R }(E_u)\le{\mathcal{E}}_{ {\mathcal{B}}_R }(E_u+\Phi)$.
\end{lemma}

\begin{proof} The argument is by contradiction. 
We suppose that there exist~$R>0$
and a perturbation~$\Phi\in C^\infty_0({\mathcal{B}}_R)$ with
$$ (E_u+\Phi)(X)\in \big[ E_{\underline u}(X),\,E_{\overline u}(X)\big]\qquad{\mbox{for 
any }}X\in\R^{n+1}_+$$
and such that
$$ {\mathcal{E}}_{ {\mathcal{B}}_R }(E_u+\Phi) < {\mathcal{E}}_{ {\mathcal{B}}_R }(E_u).$$
That is, letting~$\varphi(x):=\Phi(x,0)$, 
by Proposition~\ref{0OAPQO182:P},
\begin{eqnarray*}
0 &>& 
\frac{1}{2}\int_{\R^{n+1}_+}z^a \big(|\nabla (E_u+\Phi)(x,z)|^2-
|\nabla E_u(x,z)|^2\big)\,dx\,dz
+\int_{B_R} \big( F((u+\varphi)(x))-F(u(x))\big)\,dx
\\ &=&
\frac{1}{2}\iint_{\R^{2n}}\frac{|(u+\varphi)(x)-(u+\varphi)(y)|^2
-|u(x)-u(y)|^2}{|x-y|^{n+2s}}\,dx\,dy
+\int_{B_R} \big( F((u+\varphi)(x))-F(u(x))\big)\,dx\\&=&
{\mathcal{F}}_{B_R}(u+\varphi)-{\mathcal{F}}_{B_R}(u).
\end{eqnarray*}
Therefore, there exists a perturbation~$w_o:=u+\varphi_o$ of~$u$, with~$\varphi_o\in C^\infty_0(B_R)$,
such that
\begin{equation}\label{TRA:1}
w_o(x)\in \big[ {\underline u}(x'),\,{\overline u}(x')\big]\qquad{\mbox{for 
any }}x\in\R^n\end{equation}
and
\begin{equation}\label{GHL:ocon} {\mathcal{F}}_{B_R}(w_o)=\inf_{{\varphi\in C^\infty_0(B_R)}\atop{
(u+\varphi)(x)
\in [ {\underline u}(x'),\,{\overline u}(x')]}}
{\mathcal{F}}_{B_R}(u+\varphi)<{\mathcal{F}}_{B_R}(u).\end{equation}
As a consequence, by taking energy perturbations,
we see that~$(-\Delta)^s w_o=f'(w_o)$ inside~$\{
x\in\R^n {\mbox{ s.t. }} w_o(x)\in [ {\underline u}(x'),\,{\overline u}(x')]\}$.
In addition, if~$w(x_o)={\overline u}(x_o)$, then~$(-\Delta)^s w_o(x_o)\le f'(w_o(x_o))$.
Similarly,
if~$w(x_o)={\underline u}(x_o)$, then~$(-\Delta)^s w_o(x_o)\ge f'(w_o(x_o))$.

Now we claim that strict inequalities hold in~\eqref{TRA:1}, namely
\begin{equation}\label{TRA:2}
w_o(x)\in \big( {\underline u}(x'),\,{\overline u}(x')\big)\qquad{\mbox{for 
any }}x\in\R^n.\end{equation}
To check this, suppose by contradiction, for instance,
that there exists~$x_o\in\R^n$
such that~$w_o(x_o)={\overline u}(x_o')$. Then, the function~$\zeta(x):=
{\overline u}(x')-w_o(x)$ has a minimum at~$x_o$. Accordingly,
$$ 0\ge(-\Delta)^s \zeta(x_o) \ge f({\overline u}(x_o'))- f'(w_o(x_o)) =0.$$
This gives that~$\zeta$ must vanish identically, and thus that~$w_o$ is identically equal to~$
{\overline u}$. Then, taking~$\bar x\in\R^n\setminus B_R$, we have that
$$ {\overline u}(\bar x')= w_o(\bar x)=u(\bar x)<
{\overline u}(\bar x').$$
This is a contradiction, and therefore~\eqref{TRA:2} is established.

{F}rom~\eqref{TRA:2}, it follows that
$$ \max_{\overline{B_R}}( w_o-{\overline u})<0.$$
Now, we let~$w_k(x):=u(x+ke_n)$. We claim that there exists~$\bar k\in\N$ such that
\begin{equation} \label{df:29}
w_{\bar k}(x)>w_o(x) {\mbox{ for any }} x\in\R^n.\end{equation}
To prove this, we argue by contradiction and suppose that for any~$k\in\N$
there exists~$x_k\in\R^n$ such that~$
w_{k}(x_k)\le w_o(x_k)$.
Then, $x_k\in B_R$
(otherwise~$w_o(x_k)=u(x_k)<u(x_k+ke_n)=w_k(x_k)$).
Accordingly, up to a subsequence, we may suppose that~$x_k\to\bar x$,
for some~$\bar x\in\overline{B_R}$ as~$k\to+\infty$.
Consequently,
$$ 0\le\lim_{k\to+\infty} w_o(x_k)-w_k(x_k)
= w_o(\bar x)-
\lim_{k\to+\infty} u(x_k-ke_n)=
w_o(\bar x)-
\lim_{k\to+\infty} u(\bar x-ke_n)=
w_o(\bar x)-\overline u(\bar x').
$$
But this inequality is in contradiction with~\eqref{TRA:2}
and thus we have proved~\eqref{df:29}.

Now, starting from~\eqref{df:29}, we can reduce~$k$ till
a touching between~$w_k$ and~$w_o$ occurs.
That is, we define~$k_\star:=\inf\{ k\in [0,\bar k]
{\mbox{ s.t. }}w_k>w_o
\}$. We claim that
\begin{equation}\label{k0}
k_\star=0.
\end{equation}
Once again, suppose not. Then, the function~$v:=w^{k_\star}-w_o$
would satisfy~$v\ge0$ in~$\R^n$, with~$v(\tilde x)=0$, for some~$\tilde x\in\overline{B_R}$.
As a consequence,
$$ 0\ge (-\Delta)^s v(\tilde x) = f(w^{k_\star}(\tilde x))-f(w(\tilde x))=0,$$
which implies that~$v$ vanishes identically. In particular, fixing~$x_\star$
outside~$B_R$, we would have that
$$ 0=v(x_\star)=
w^{k_\star}(x_\star)-w_o(x_\star)=
u(x_\star+k_\star e_n)-u(x_\star)>0.$$
This contradiction completes the proof of~\eqref{k0}.

Now, in view of~\eqref{k0}, we obtain that, for any~$x\in\R^n$,
$$ w_o(x)\le\lim_{k\searrow0} w_k(x)=
\lim_{k\searrow0} u(x+ke_n)=u(x).$$
Similarly, one can prove that~$w_o(x)\ge u(x)$.
Therefore, $w_o$ and~$u$ must coincide
and so~${\mathcal{F}}_{B_R}(w_o)={\mathcal{F}}_{B_R}(u)$.
But this fact is in contradiction with~\eqref{GHL:ocon}
and so we have completed the proof of Lemma~\ref{INTRAP}.
\end{proof}

\section{Local minimization properties inherited from those of the profiles at infinity}\label{CLASS:C}

In this section, we show that if the profiles at infinity are local minimizers,
then so is the original solution. In the classical case of the Laplacian, this property
was discussed, for instance, in Proposition~2.3 of~\cite{FV11}.
In our setting, the result that
we need is the following (and it uses
the pivotal definition of extended local minimizer in Definition~\ref{D:L:2}):

\begin{lemma}\label{0tyg27427412}
Let~$s\in(0,1)$ and~$u\in C^2(\R^n,[-1,1])$
be a solution of~$(-\Delta)^s u=f(u)$ in~$\R^n$, 
with~$\partial_{x_n}u>0$ in~$\R^n$.
Let
$$ \underline u(x'):=\lim_{x_n\to-\infty} u(x',x_n)\quad{\mbox{ and }}\quad
\overline u(x'):=\lim_{x_n\to+\infty} u(x',x_n)$$
and suppose that~$\underline u$ and~$\overline u$ are local minimizers (in~$\R^{n-1}$).
Then, $E_u$ is an extended local minimizer (in~$\R^{n+1}_+$)
and~$u$ is a local minimizer (in~$\R^{n}$).
\end{lemma}

\begin{proof} Our goal is to show that $E_u$ is an extended local minimizer 
in the sense of Definition~\ref{D:L:2} (from this, we also obtain that~$u$ is a local minimizer,
thanks to Proposition~\ref{EQUI DL}). To this aim, fixed~$x_n\in\R$,
we use the following ``slicing notation'' for a domain~$\Omega\subset\R^{n+1}$:
we let
$$ \Omega^{x_n} := \{ (x',z)\in\R^{n-1}\times\R {\mbox{ s.t. }}
(x',x_n,z)\in\Omega\}.$$
Also, the function~$\overline u:\R^{n-1}\to\R$ can be seen as a function on~$\R^n$,
by defining~$\overline u_\star(x',x_n):=\overline u(x')$ and so we can consider its $a$-harmonic
extension~$E_{\overline u_\star}$.
Given~$R>0$ and~$\Phi\in C^\infty_0({\mathcal{B}}_R)$, with~$\varphi(x):=\Phi(x,0)$,
we also define~$\Phi_{x_n}(x',z):=\Phi(x',x_n,z)$ and~$\varphi_{x_n}(x'):=\varphi(x',x_n)$.
Hence
we have that
\begin{equation}\label{8920tr}
\begin{split}
{\mathcal{E}}_{ {\mathcal{B}}_R }(E_{\overline u_\star}+\Phi)\;
&= 
\frac{1}{2}
\int_{{\mathcal{B}}_R^+}z^a |\nabla (E_{\overline u_\star}+\Phi)(x,z)|^2\,dx\,dz+
\int_{B_R} F(\overline u(x')+\varphi(x))\,dx
\\ &\ge \int_\R\left[
\frac{1}{2}
\int_{(B_R)_{x_n}\times(0,R)}z^a |\nabla_{(x',z)} E_{\overline u}(x',z)+\nabla_{(x',z)}
\Phi_{x_n} (x',z)|^2\,dx'\,dz\right.\\ &\qquad+\left.
\int_{(B_R)_{x_n}} F(\overline u(x')+\varphi_{x_n}(x))\,dx'
\right]\,dx_n
\end{split}\end{equation}
Also, since~$\overline u$ is a local minimizer in~$\R^{n-1}$, we have
that~$E_{\overline u}$ is an extended local minimizer in~$\R^{n-1}\times(0,+\infty)$,
thanks to
Proposition~\ref{EQUI DL}, and therefore
\begin{eqnarray*}&&
\frac{1}{2}
\int_{(B_R)_{x_n}\times(0,R)}z^a |\nabla_{(x',z)} E_{\overline u}(x',z)+\nabla_{(x',z)}
\Phi_{x_n} (x',z)|^2\,dx'\,dz+
\int_{(B_R)_{x_n}} F(\overline u(x')+\varphi_{x_n}(x))\,dx'
\\ &\ge&
\frac{1}{2}
\int_{(B_R)_{x_n}\times(0,R)}z^a |\nabla_{(x',z)} E_{\overline u}(x',z)|^2\,dx'\,dz
+\int_{(B_R)_{x_n}} F(\overline u(x'))\,dx'.
\end{eqnarray*}
By inserting this inequality into~\eqref{8920tr}, we obtain that~$
{\mathcal{E}}_{ {\mathcal{B}}_R }(E_{\overline u_\star}+\Phi)\ge
{\mathcal{E}}_{ {\mathcal{B}}_R }(E_{\overline u_\star})$.
That is, 
\begin{equation}\label{OVER}
{\mbox{$E_{\overline u_\star}$ is an extended local minimizer in~$\R^{n+1}_+$.
}}\end{equation}
Similarly,
one can define~$\underline u_\star(x',x_n):=\underline u(x')$
and conclude that
\begin{equation}\label{UNDER}
{\mbox{$E_{\underline u_\star}$ is an extended local minimizer in~$\R^{n+1}_+$.
}}\end{equation}
Now,
given~$R>0$ and~$\Psi\in C^\infty_0({\mathcal{B}}_R)$, with~$\psi(x):=\Psi(x,0)$,
we consider the perturbation~$E_u+\Psi$.
We define
\begin{eqnarray*}
&&\alpha(X):=\left\{
\begin{matrix}
E_{\overline u_\star}(X) & {\mbox{ if }} (E_u+\Psi)(X)\le E_{\overline u_\star}(X),\\
(E_u+\Psi)(X) & {\mbox{ if }} (E_u+\Psi)(X)> E_{\overline u_\star}(X),
\end{matrix}
\right. \\
&&\beta(X):=\left\{
\begin{matrix}
E_{\overline u_\star}(X) & {\mbox{ if }} (E_u+\Psi)(X)> E_{\overline u_\star}(X),\\
E_{\underline u_\star}(X) & {\mbox{ if }} (E_u+\Psi)(X)< E_{\underline u_\star}(X),\\
(E_u+\Psi)(X) & {\mbox{ if }} (E_u+\Psi)(X)\in\big[ E_{\underline u_\star}(X),
\,E_{\overline u_\star}(X)\big],
\end{matrix}
\right. \\
{\mbox{and }}&&\gamma(X):=\left\{
\begin{matrix}
E_{\underline u_\star}(X) & {\mbox{ if }} (E_u+\Psi)(X)\ge E_{\underline u_\star}(X),\\
(E_u+\Psi)(X) & {\mbox{ if }} (E_u+\Psi)(X)< E_{\underline u_\star}(X).
\end{matrix}
\right. 
\end{eqnarray*}
By~\eqref{OVER}, we have that
\begin{equation}\label{P:X1}
{\mathcal{E}}_{ \{ E_u+\Psi>E_{\overline u_\star}\} }(E_{\overline u_\star}) \le 
{\mathcal{E}}_{ \{ E_u+\Psi>E_{\overline u_\star}\} }(\alpha).
\end{equation}
Similarly,
by~\eqref{UNDER}, we have that
\begin{equation}\label{P:X2}
{\mathcal{E}}_{ \{ E_u+\Psi<E_{\underline u_\star}\} }(E_{\underline u_\star}) \le 
{\mathcal{E}}_{ \{ E_u+\Psi<E_{\underline u_\star}\} }(\gamma).
\end{equation}
In addition, from Lemma~\ref{INTRAP},
\begin{equation}\label{P:X3}
{\mathcal{E}}_{{\mathcal{B}}_R }(E_u) \le 
{\mathcal{E}}_{{\mathcal{B}}_R}(\beta).
\end{equation}
Notice also that~$ \{ E_u+\Psi>E_{\overline u_\star}\}$ and~$
\{ E_u+\Psi<E_{\underline u_\star}\}$ are subset of~$ {{\mathcal{B}}_R } $.
Thus, using~\eqref{P:X1}, \eqref{P:X2} and~\eqref{P:X3},
\begin{eqnarray*}
{\mathcal{E}}_{{\mathcal{B}}_R }(E_u+\Psi)&=&
{\mathcal{E}}_{\{ E_u+\Psi>E_{\overline u_\star}\} }(E_u+\Psi)+
{\mathcal{E}}_{\{ E_u+\Psi<E_{\underline u_\star}\} }(E_u+\Psi)+
{\mathcal{E}}_{{\mathcal{B}}_R\cap
\{ E_{\underline u_\star}\le E_u+\Psi\le E_{\overline u_\star}\} }(E_u+\Psi)\\
&=&
{\mathcal{E}}_{\{ E_u+\Psi>E_{\overline u_\star}\} }(\alpha)+
{\mathcal{E}}_{\{ E_u+\Psi<E_{\underline u_\star}\} }(\gamma)+
{\mathcal{E}}_{{\mathcal{B}}_R\cap
\{ E_{\underline u_\star}\le E_u+\Psi\le E_{\overline u_\star}\} }(\beta)\\
&\ge& 
{\mathcal{E}}_{\{ E_u+\Psi>E_{\overline u_\star}\} }(E_{\overline u_\star})+
{\mathcal{E}}_{\{ E_u+\Psi<E_{\underline u_\star}\} }(E_{\underline u_\star})+
{\mathcal{E}}_{{\mathcal{B}}_R\cap
\{ E_{\underline u_\star}\le E_u+\Psi\le E_{\overline u_\star}\} }(\beta)\\
&=&
{\mathcal{E}}_{\{ E_u+\Psi>E_{\overline u_\star}\} }(\beta)+
{\mathcal{E}}_{\{ E_u+\Psi<E_{\underline u_\star}\} }(\beta)+
{\mathcal{E}}_{{\mathcal{B}}_R\cap
\{ E_{\underline u_\star}\le E_u+\Psi\le E_{\overline u_\star}\} }(\beta)\\
&=&{\mathcal{E}}_{{\mathcal{B}}_R }(\beta)\\
&\ge& {\mathcal{E}}_{{\mathcal{B}}_R }(E_u).
\end{eqnarray*}
This shows that~$E_u$ is an extended local minimizer, as desired.
\end{proof}

\section{Proof of Theorem~\ref{MAIN}}\label{CLASS:D}

By Lemma~\ref{02738}, we know that
both~$\underline u$ and~$\overline u$ are local minimizers.
This and Lemma~\ref{0tyg27427412} imply that~$u$ is a local minimizer.
Therefore, by Lemma 8.1 in~\cite{DSVarxiv},
we have that~$u_\varepsilon(x):=u(x/\varepsilon)$ approaches, as~$\varepsilon\to0$, up to subsequences,
a step function of the form~$\chi_E-\chi_{E^c}$, and the level sets of~$u_\varepsilon$ approach~$\partial E$
locally 
uniformly. {We claim that
\begin{equation}\label{89:013948548}
{\mbox{$E$ is a halfspace.
}} 
\end{equation}
In a sense, this claim is a version of the Bernstein-type result in~\cite{figalli},
for which we provide a complete and independent proof, by arguing as follows.
To prove~\eqref{89:013948548}, it is enough to show that
\begin{equation}\label{89:013948548:E}
{\mbox{either~$E$ is contained in a halfspace, or it contains a halfspace,
}} 
\end{equation}
see e.g. Lemma~8.3 in~\cite{DSVarxiv}, hence we focus on the proof of~\eqref{89:013948548:E}.
To this end, we distinguish two cases,
\begin{equation}\label{CA:-1}
{\mbox{either~$\underline u=-1$ and~$\overline u=1$,}}
\end{equation}
or 
\begin{equation}\label{CA:-2}
{\mbox{at least one between~$\underline u$ and~$\overline u$ is non-constant.
}}\end{equation}
In case~\eqref{CA:-1}, we can exploit Theorem~1.4 in~\cite{DSVarxiv}
and obtain in particular that~\eqref{89:013948548:E} holds true,
so we focus on case~\eqref{CA:-2} and we suppose that~$\underline u$
is non-constant (the case in which~$\overline u$ is non-constant is similar).

Then, setting~$\underline{u}_\varepsilon(x):=\underline u(x/\varepsilon)$
we have that~$\underline{u}_\varepsilon$ approaches,
as~$\varepsilon\to0$, up to subsequences,
a step function of the form~$\chi_{\underline E}-\chi_{{\underline E}^c}$.
Since~$\underline{u}$ is a non-constant $1$D function, we have that~${\underline E}$ is a halfspace.
In addition, a.e.~$x\in\R^n$,
$$ (\chi_{\underline E}-\chi_{{\underline E}^c})(x)=
\lim_{\varepsilon\searrow0} \underline u(x/\varepsilon)\le
\lim_{\varepsilon\searrow0} u(x/\varepsilon)=
(\chi_{E}-\chi_{{E}^c})(x),$$
and consequently~${\underline E}\subseteq E$.
This proves~\eqref{89:013948548:E}, and so~\eqref{89:013948548}.
}

Hence,
$u$ is necessarily $1$D, thanks to~\eqref{89:013948548}
and Theorem~1.2 in~\cite{DSVarxiv}.

\section{The case of two-dimensional profiles at infinity}\label{UL810}

For completeness, in relation with the work in~\cite{MR3107529},
we observe that our arguments provide
also the following variation of Theorem~\ref{MAIN}:

\begin{theorem}\label{MAIN:VAR}
Let~$s\in\left(0,\frac12\right)$ and~$u\in C^2(\R^n,[-1,1])$
be a solution of~$(-\Delta)^s u=u-u^3$ in~$\R^n$, 
with~$\partial_{x_n}u>0$ in~$\R^n$.

Let 
$$ \underline u(x'):=\lim_{x_n\to-\infty} u(x',x_n)\quad{\mbox{ and }}\quad
\overline u(x'):=\lim_{x_n\to+\infty} u(x',x_n).$$
Assume that (possibly after a rotation) 
$\underline u$ and~$\overline u$ depend on at most two Euclidean variables
(not necessarily the same).

Then $u$ is a local minimizer.

Moreover, if $n\le8$, there exists~$s(n)\in \left[0,\frac12\right)$ such that if~$
s\in\left(s(n),\frac12\right)$ then~$u$ is $1$D.
\end{theorem}

The proof of Theorem~\ref{MAIN:VAR} shares
the point of view taken in~\cite{MR2483642} and~\cite{FV11},
and follows the same lines
as that of Theorem~\ref{MAIN}, with the modifications listed here below:
\begin{itemize}
\item By
following verbatim the proof of Lemma~\ref{02738}, and
exploiting that~$\underline u$ and~$\overline u$ are stable two-dimensional solutions, one obtains that 
they are local minimizers and~$1$D;
\item {F}rom this and Lemma~\ref{0tyg27427412}, one deduces
the first claim in Theorem~\ref{MAIN:VAR};
\item The second claim in Theorem~\ref{MAIN:VAR} follows
from the first claim and the argument in Section~\ref{CLASS:D}
(in this framework, for $s$ large enough,
one can exploit
Theorem~1.6 of~\cite{DSVarxiv}
in place of Theorem~1.4 of~\cite{DSVarxiv}).
\end{itemize}

\vfill
\end{document}